\documentclass[reqno, 12pt]{article}

\pdfoutput=1

\usepackage{enumerate}
\usepackage{latexsym}
\usepackage[centertags]{amsmath}
\usepackage{amsfonts}
\usepackage{amssymb}
\usepackage{amsthm}
\usepackage{newlfont}
\usepackage{graphics}
\usepackage{color}
\usepackage{float}
\usepackage{diagbox}
\usepackage[linesnumbered,ruled,vlined]{algorithm2e}
\usepackage{hyperref}
\usepackage{longtable}
\usepackage{rotating}
\usepackage{multirow}
\usepackage{extarrows}
\usepackage[sort,compress,numbers]{natbib}
\usepackage[utf8]{inputenc}

\newtheorem{thm}{Theorem}[section]
\newtheorem{cor}[thm]{Corollary}
\newtheorem{lem}[thm]{Lemma}

\theoremstyle{mydefinition}

\theoremstyle{myremark}

\allowdisplaybreaks[4]
\SetKwInput{KwInput}{Input}                
\SetKwInput{KwOutput}{Output}              

\renewcommand{\P}{{\mathbb{P}}}

\title{Enumeration of Corona for Lozenge Tilings}

\author{Craig Knecht$^{\color{blue} \S}$, Feihu Liu$^{\color{blue} \dag}$, and Guoce Xin$^{\color{blue} \P}$
\\[2mm]
{\small $^{\color{blue} \S}$ 691 Harris Lane, Gallatin, Tennessee 37066, USA}\\
{\small $^{\color{blue} \dag, \P}$ School of Mathematical Sciences,}\\[-0.8ex]
{\small Capital Normal University, Beijing, 100048, P.R.~China}\\
{\small {\color{blue} $^\S$} Email address: craigknecht03@gmail.com}\\
{\small {\color{blue} $^\dag$} Email address: liufeihu7476@163.com}\\
{\small {\color{blue} $^\P$} Email address: guoce\_xin@163.com}
}

\date{April 27, 2025}

\begin{document}

\maketitle

\begin{abstract}
Knecht considers the enumeration of coronas. This is a counting problem for two specific types of lozenge tilings. Their exact closed formulas are conjectured in [A380346] and [A380416] on the OEIS. We prove this conjecture by using the weighted adjacency matrix. Furthermore, we extend this result to a more general setting.
\end{abstract}

\noindent
\begin{small}
\end{small}

\noindent
\begin{small}
\end{small}

\section{Introduction}
We first introduce some basic definitions (for instance, see \cite{Fulmek21}). 
The \emph{triangular lattice} can be regarded as a tiling of the Euclidean plane $\mathbb{R}^2$ using unit equilateral triangles. A \emph{lozenge} (or \emph{diamond}) is a union of any two unit equilateral triangles sharing an edge.
A \emph{region} in the triangular lattice is defined as a finite subset of triangles.
A \emph{lozenge tiling} of such a region $R$ is a partition of $R$ into blocks, where each block is a lozenge; that is to say, a lozenge tiling of a region $R$ is a covering of the region $R$ by lozenges so that there are no gaps or overlaps.

There are precisely two possible orientations of the unit equilateral triangles, i.e., upwards-pointing and downwards-pointing. Therefore, there are three possible orientations for lozenges, namely left-tilted, right-tilted, and vertical. See Figure \ref{Tu1}.
\begin{figure}[htp]
\centering
\includegraphics[width=7.5cm,height=2.5cm]{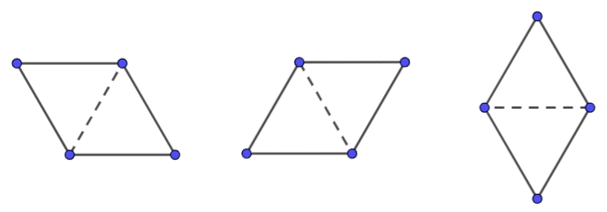}
\caption{Left-tilted, right-tilted, and vertical.}
\label{Tu1}
\end{figure}

The basic problem in this context is the enumeration of all lozenge tilings for some given region $R$.
This is a well-studied topic in combinatorics. It has attracted considerable interest due to the elegance of the formulas and the ingenuity of the combinatorial arguments, (see, for example, \cite{CiucuLai19,Fulmek21,Lai17,Rohatgi15}).
Many methods for tiling enumerations have been developed and studied. For instance, Kuo's graphical condensation method \cite{Kuo04}. Ciucu, Eisenk\"olbl, Krattenhaler, and Zare \cite{CEKZ01} provide closed formulas for certain lozenge tilings by employing the theory of nonintersecting lattice paths and determinant evaluations.
Lai and his co-author investigated tiling enumerations in a series of papers, (see, for instance, \cite{Lai14,Lai15,Lai-Rohatgi19,Lai21,Lai-Rohatgi21,CiucuLai20}).

The motivation for this work comes from the two recent studies on enumeration of corona \cite[A380346, A380416]{Sloane23}.
The word ``corona" was first used by Knecht. Given a regular hexagon $H$, a \emph{corona of a hexagon $H$} is a lozenge tiling along the edges of $H$ such that no additional lozenges are utilized. For example, see Figure \ref{Tu2}, Figure $a$ is a corona of a hexagon $H$ with side length $1$; Figure $b$ is a corona of a hexagon $H$ with length $2$; but Figure $c$ is not a corona since the lozenge $e$ is redundant; Figure $d$ is also not a corona because a lozenge is missing in the bottom left corner. The number of coronas of a hexagon $H$ with side length $n$ is denoted by $H(n)$.
\begin{figure}[htp]
\centering
\includegraphics[width=16cm,height=5cm]{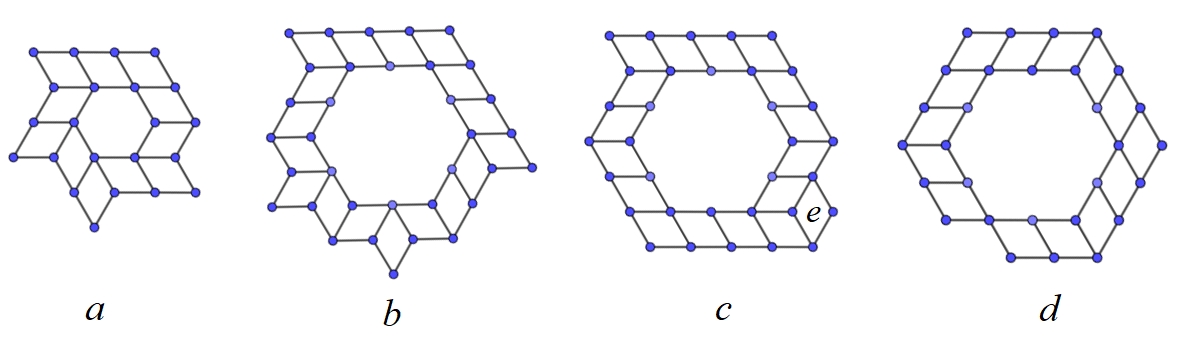}
\caption{Corona of a hexagon $H$.}
\label{Tu2}
\end{figure}

The first result of this paper is to give an exact closed formula for $H(n)$. In other words, we prove the following conjecture.  

\begin{thm}{\em (Conjectured in \cite[A380346]{Sloane23})}\label{Them-A380346}
Let $n\in \mathbb{N}$. Let $H(n)$ be the number of coronas of a hexagon $H$ with side length $n$.
Then there are only four cases in which the number of lozenges is used in a corona of a hexagon $H$ with side length $n$, namely $6n+3$, $6n+4$, $6n+5$, and $6n+6$.
Let $h_i(n)$ be the number of corona tilings for $6n+2+i$ with $1\leq i\leq 4$. Then we have
$$h_1(n)=2,\ \ \ h_2(n)=9(n+1)^2,\ \ \ h_3(n)=6(n+1)^4,\ \ \ h_4(n)=(n+1)^6.$$
Furthermore, we obtain
\begin{align*}
H(n)=h_1(n)+h_2(n)+h_3(n)+h_4(n)=n^6+6n^5+21n^4+44n^3+60n^2+48n+18.
\end{align*}
\end{thm}

Similarly, when we are given a diamond $D$ with $60^\circ$ and $120^\circ$ angles, a \emph{corona of a diamond $D$} is a lozenge tiling along the edges of $D$ such that no additional lozenges are utilized. For example, see Figure \ref{Tu3}, Figure $p$ is a corona of a diamond $D$ with side length $1$; Figure $q$ is a corona of a diamond $D$ with side length $2$; but Figure $u$ is not a corona since the lozenge $w$ is redundant; Figure $v$ is also not a corona because a lozenge is missing in the bottom right corner. The number of coronas of a diamond $D$ with side length $n$ is denoted by $D(n)$.
\begin{figure}[htp]
\centering
\includegraphics[width=16cm,height=6cm]{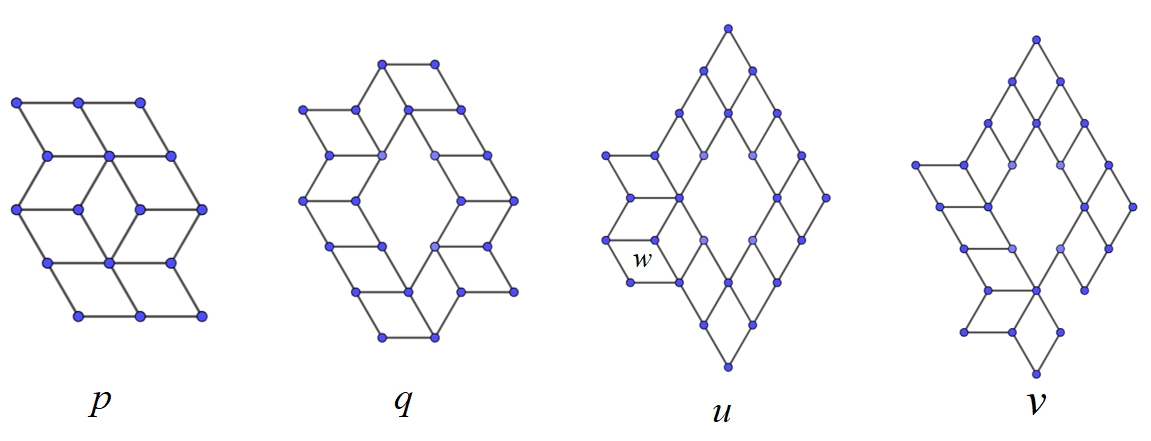}
\caption{Corona of a diamond $D$.}
\label{Tu3}
\end{figure}

Our second result is to provide an exact closed formula for $D(n)$.

\begin{thm}{\em (Conjectured in \cite[A380416]{Sloane23})}\label{Them-A380416}
Let $n\in \mathbb{N}$. Let $D(n)$ be the number of coronas of a diamond $D$ with side length $n$.
Then there are only four cases in which the number of lozenges is used in a corona of a diamond $D$ with side length $n$, namely $4n+3$, $4n+4$, $4n+5$, and $4n+6$.
Let $d_i(n)$ be the number of corona tilings for $4n+2+i$ with $1\leq i\leq 4$. Then we have
$$d_1(n)=2,\ \ \ d_2(n)=(2n+3)^2,\ \ \ d_3(n)=2(n+1)^2(2n+3),\ \ \ d_4(n)=(n+1)^4.$$
Furthermore, we get
\begin{align*}
D(n)=d_1(n)+d_2(n)+d_3(n)+d_4(n)=n^4+8n^3+24n^2+32n+18.
\end{align*}
\end{thm}
 
Inspired by Knecht's work, we study the extension for the corona. We alter the side lengths of the regular hexagon $H$ and the diamond $D$. We only require that the opposite sides of the hexagon and diamond remain equal in length, which ensures that their interior angles remain unchanged. The modified hexagon and diamond are denoted by $\overline{H}$ and $\overline{D}$, respectively; see Figure \ref{Tu13}. Our third result provides an exact closed formula for the number of coronas of a hexagon $\overline{H}$ and diamond $\overline{D}$, respectively; see Theorems \ref{ExtensionOne} and \ref{ExtensionTwo}. These are generalizations of Theorem \ref{Them-A380346} and Theorem \ref{Them-A380416}, respectively.

This paper is organized as follows. In Section 2, we introduce some results on the enumeration of walks in graphs. Section 3 is devoted to the proof of the main theorems. Finally, we extend Theorems \ref{Them-A380346} and \ref{Them-A380416} in Section 4.

\section{Preliminary: Walks in Graphs}

The goal of this section is to describe the number of walks in a graph (for instance, see \cite[Chapter 1]{RP.Stanley-AC}), which will later be useful to prove our main theorems.

A \emph{graph} $G=(V,E)$ consists of a \emph{vertex set} $V=\{v_1,v_2,\ldots,v_m\}$ and an \emph{edge set} $E$, where an edge is an unordered pair of vertices of $G$.
We will use the same notation as in \cite[Chapter 1]{RP.Stanley-AC}.
The \emph{adjacency matrix} of the graph $G$ is the $m\times m$ matrix $A(G)$, whose $(i,j)$-entry $a_{ij}$ is equal to the number of edges incident to $v_i$ and $v_j$. Thus $A(G)$ is a symmetric matrix.

A \emph{walk} in $G$ of \emph{length} $\ell$ from vertex $u$ to vertex $v$ is a sequence $v_1,e_1,v_2,e_2,\ldots,v_{\ell},e_{\ell},v_{\ell+1}$ such that $v_i\in V$, $e_j\in E$, the vertices of $e_i$ are $v_i$ and $v_{i+1}$, for $1\leq i\leq \ell$, and $v_1=u$, $v_{\ell}=v$.

\begin{lem}{\em (\cite[Theorem 1.1]{RP.Stanley-AC})}\label{WalksinG}
For any positive integer $\ell$, the $(i,j)$-entry of the matrix $A(G)^{\ell}$ is equal to the number of walks from $v_i$ to $v_j$ in $G$ of length $\ell$.
\end{lem}

A \emph{closed walk} in $G$ is a walk that ends where it begins. The number of closed walks in $G$ of length $\ell$ starting at $v_i$ is therefore given by $(A(G)^{\ell})_{ii}$.

\begin{cor}{\em (\cite[Chapter 1]{RP.Stanley-AC})}\label{Close-walk}
The total number $f_G(\ell)$ of closed walks of length $\ell$ is given by 
$$f_G(\ell)=\sum_{i=1}^m(A(G)^{\ell})_{ii}=\mathrm{tr}(A(G)^{\ell}),$$
where $\mathrm{tr}$ denotes trace (that is, the sum of main diagonal entries). 
\end{cor}

\section{Proofs of The Main Theorems}

Given a hexagon $H$ with side length $n$, the six corners of the hexagon $H$ are denoted by $1,2,3,4,5,6$; see Figure \ref{Tu4}.
By observing the corona of a hexagon $H$, we found that there are five states for lozenges at one corner of a hexagon $H$. For example, the five states at corner $1$ are shown in Figure \ref{Tu5}. The number of lozenges is 3, 3, 4, 2, and 3, respectively. The states of the remaining 5 corners are obtained by rotating the states at corner 1. For instance, the five states at corner $2$ are shown in Figure \ref{Tu6}.

\begin{figure}[htp]
\centering
\includegraphics[width=4cm,height=3.5cm]{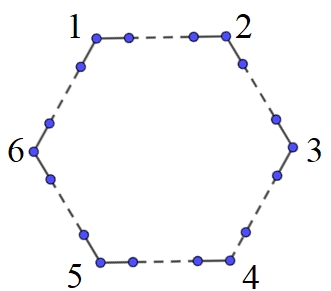}
\caption{A hexagon $H$.}
\label{Tu4}
\end{figure}

\begin{figure}[htp]
\centering
\includegraphics[width=16cm,height=4cm]{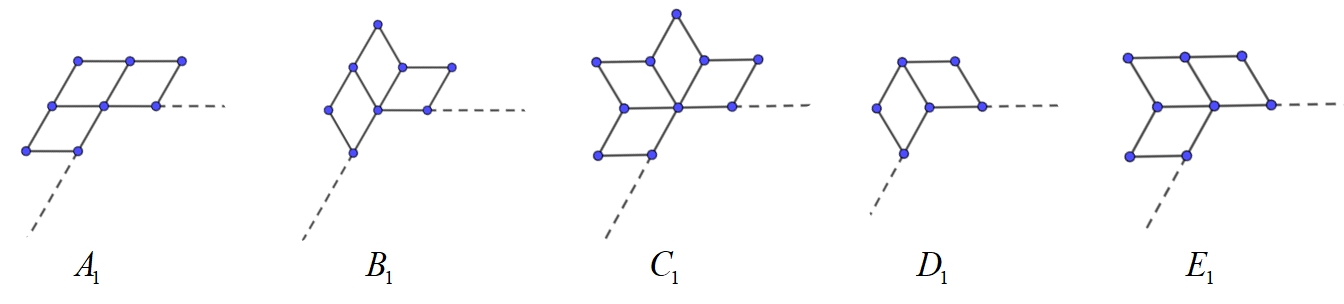}
\caption{The five states at corner $1$.}
\label{Tu5}
\end{figure}

\begin{figure}[htp]
\centering
\includegraphics[width=16cm,height=4cm]{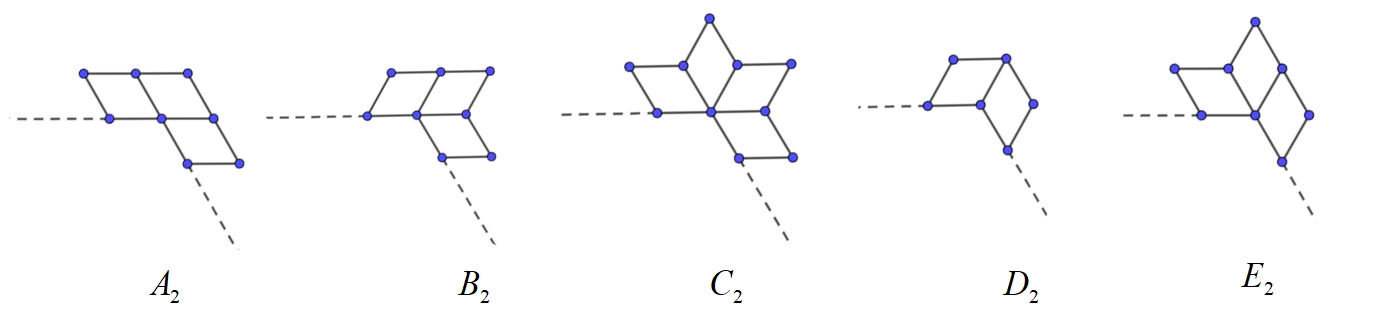}
\caption{The five states at corner $2$.}
\label{Tu6}
\end{figure}

For a corona of hexagon $H$ with side length $n$, there are $n+1$ states for lozenges between corner $1$ and corner $2$, see Figure \ref{Tu7}. The states of the other edges of the hexagon $H$ are obtained by rotating these $n+1$ states. In Figure \ref{Tu7}, states $Q$ and $K$ have $n-2$ lozenges, while state $L_i, 1\leq i\leq n-1$ has $n-1$ lozenges.
\begin{figure}[htp]
\centering
\includegraphics[width=16cm,height=9cm]{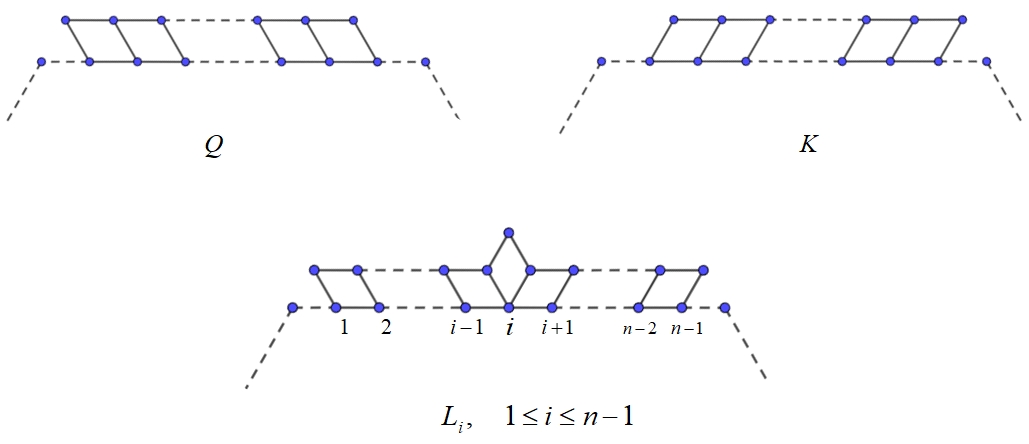}
\caption{The $n+1$ states on the side.}
\label{Tu7}
\end{figure}

Now we present our proof of Theorem \ref{Them-A380346}.
\begin{proof}[Proof of Theorem \ref{Them-A380346}]
Based on the above statement, we know that corner 1 and corner 2 are connected by the $n+1$ states in Figure \ref{Tu7}. Therefore, we can construct the following bipartite graph in Figure \ref{Tu8}. The two vertices in the graph are connected, indicating that corner 1 and corner 2 can be concatenated together through the states in Figure \ref{Tu7}. For example, the state $D_1$ can be connected to states $A_2$, $C_2$, and $E_2$ through state $Q$, and the state $D_1$ can be connected to states $B_2$ and $D_2$ through states $L_i$, $1\leq i\leq n-1$. (More precisely, there are $n-1$ multi-edges between $D_1$ and $B_2$, and also between $D_1$ and $D_2$).
\begin{figure}[htp]
\centering
\includegraphics[width=8cm,height=9cm]{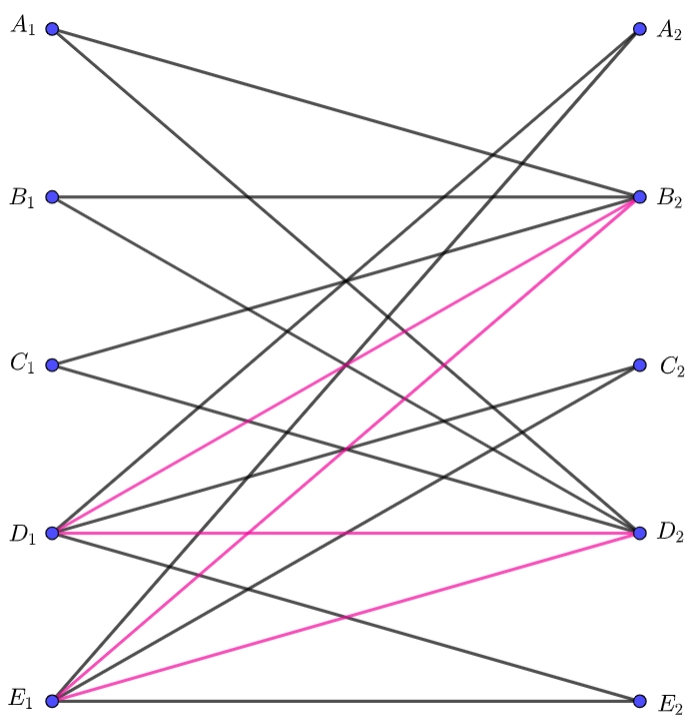}
\caption{The bipartite graph in proof of Theorem \ref{Them-A380346}.}
\label{Tu8}
\end{figure}

Now we construct the following weighted adjacency matrix $M$ of the graph in Figure \ref{Tu8}:
$$M=\begin{array}{@{}r@{}c@{}c@{}c@{}c@{}c@{}l@{}}
& A_2 & B_2 & C_2 & D_2 & E_2 \\
\left.\begin{array}{c} A_1 \\ B_1 \\ C_1 \\ D_1 \\ E_1 \end{array}\right(
& \begin{array}{c} 0 \\ 0 \\ 0 \\ x^{3+n-2} \\ x^{3+n-2}   \end{array}
& \begin{array}{c} x^{3+n-2} \\ x^{3+n-2} \\ x^{3+n-2} \\ (n-1)x^{3+n-1} \\ (n-1)x^{3+n-1} \end{array}
& \begin{array}{c}0 \\ 0 \\ 0 \\x^{4+n-2} \\ x^{4+n-2}  \end{array}
& \begin{array}{c} x^{2+n-2} \\ x^{2+n-2} \\ x^{2+n-2} \\ (n-1)x^{2+n-1} \\ (n-1)x^{2+n-1} \end{array}
& \begin{array}{c} 0 \\ 0 \\ 0 \\ x^{3+n-2} \\ x^{3+n-2} \end{array}
& \left).\begin{array}{c} \\ \\ \\ \\ \\  \end{array}\right.
\end{array},$$
where the exponent $r$ of weight $x^r$ is the number of lozenges in corner 2 and one of the edges in Figure \ref{Tu7}.
According to Corollary \ref{Close-walk}, it is clear that we only need to get the sum of the main diagonal elements of matrix $M^6$.

Through tedious matrix multiplication, it follows that 
\begin{align*}
\sum_{i=1}^5(M^{6})_{ii}=\mathrm{tr}(M^{6})=2x^{6n+3}+9(n+1)^2x^{6n+4}+6(n+1)^4x^{6n+5}+(n+1)^6x^{6n+6}.
\end{align*}
This completes the proof of Theorem \ref{Them-A380346}.
\end{proof}

We assume that readers are familiar with the relevant knowledge of generating functions (see \cite[Chapter 1]{RP.Stanley}).
\begin{cor}
Let $n\in \mathbb{N}$. Let $H(n)$ be the number of coronas of a hexagon $H$ with side length $n$. Then we obtain the generating function of $H(n)$ as follows:
$$\sum_{n\geq 0}H(n)x^n=\frac{2x^6+4x^5+114x^4+220x^3+290x^2+72x+18}{(1-x)^7}.$$
\end{cor}

\begin{figure}[htp]
\centering
\includegraphics[width=3.5cm,height=4cm]{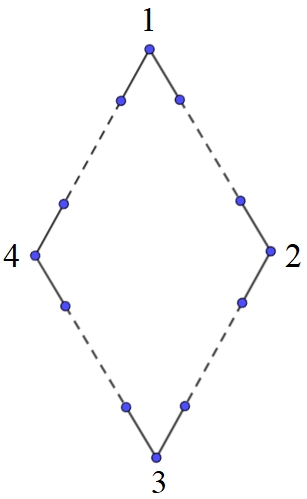}
\caption{A diamond $D$.}
\label{Tu9}
\end{figure}

Now we consider the corona of a diamond $D$. Given a diamond $D$ with side length $n$, the four corners of the diamond $D$ are denoted by $1,2,3,4$, see Figure \ref{Tu9}. By observing the corona of a diamond $D$, we found that there are 8 states for lozenges at the corner 1; see Figure \ref{Tu10}. There are 5 states for lozenges at the corner 2; see Figure \ref{Tu11}. 
The states of corner 3 and corner 4 are obtained by rotating $180^\circ$ of the states at corners 1 and 2, respectively. Similar to Figure \ref{Tu10}, the states for lozenges at corner 3 are denoted by $A_3,B_3,C_3,D_3,E_3,F_3,I_3,J_3$.
For a corona of diamond $D$ with side length $n$, which is the same as Figure \ref{Tu7}, there are $n+1$ states for lozenges between corner $1$ and corner $2$.

\begin{figure}[htp]
\centering
\includegraphics[width=14cm,height=7cm]{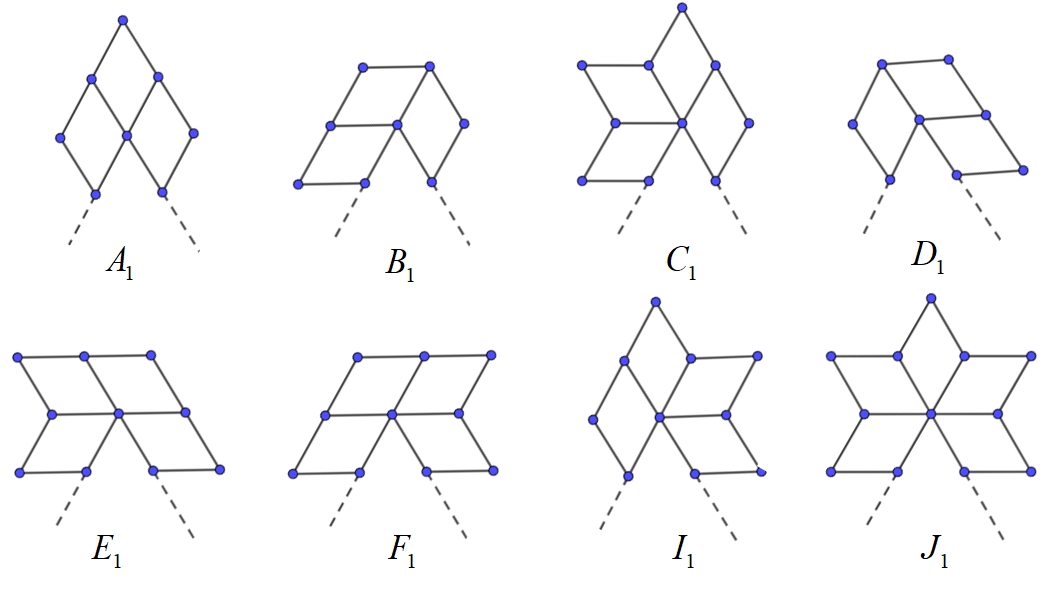}
\caption{The eight states at corner $1$.}
\label{Tu10}
\end{figure}

\begin{figure}[htp]
\centering
\includegraphics[width=15cm,height=4cm]{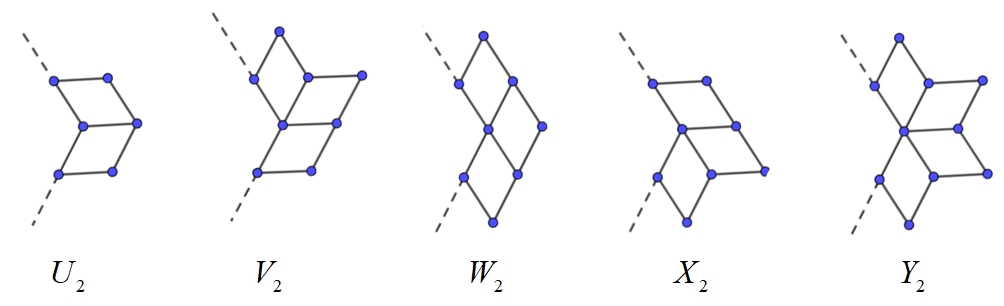}
\caption{The five states at corner $2$.}
\label{Tu11}
\end{figure}

Now we present our proof of Theorem \ref{Them-A380416}.
\begin{proof}[Proof of Theorem \ref{Them-A380416}]
Similar to the proof of Theorem \ref{Them-A380346}, we can construct the graph in Figure \ref{Tu12}.
The two vertices in the graph are connected, indicating that corner 1 and corner 2 (or corner 2 and corner 3) can be concatenated together through the states in Figure \ref{Tu7}.
For example, the state $A_1$ can be connected to states $V_2$, $W_2$, and $Y_2$ through state $Q$, and the state $A_1$ can be connected to states $U_2$ and $X_2$ through states $L_i$, $1\leq i\leq n-1$. (More precisely, there are $n-1$ multi-edges between $A_1$ and $U_2$, and also between $A_1$ and $X_2$).
\begin{figure}[htp]
\centering
\includegraphics[width=15cm,height=13cm]{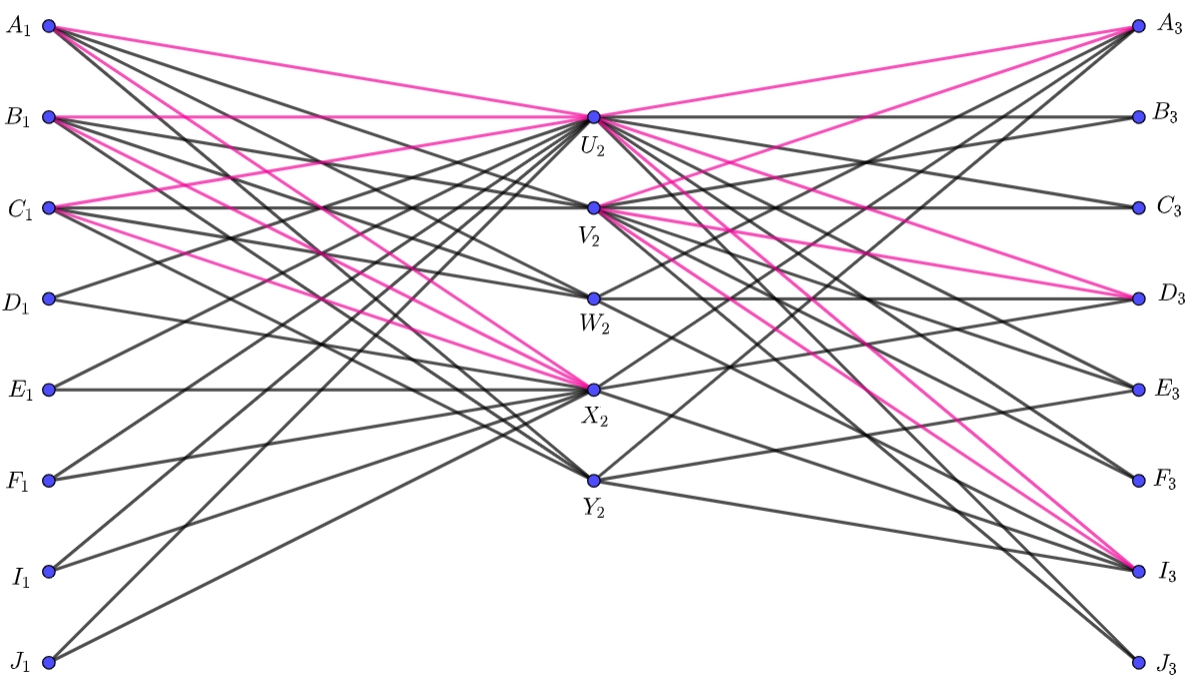}
\caption{The graph in proof of Theorem \ref{Them-A380416}.}
\label{Tu12}
\end{figure}

Now we construct the following two weighted adjacency matrices $R$ and $T$:
$$R=\begin{array}{@{}r@{}c@{}c@{}c@{}c@{}c@{}l@{}}
& U_2 & V_2 & W_2 & X_2 & Y_2 \\
\left.\begin{array}{c} A_1 \\ B_1 \\ C_1 \\ D_1 \\ E_1 \\ F_1\\ I_1\\ J_1 \end{array}\right(
& \begin{array}{c} (n-1)x^{2+n-1} \\ (n-1)x^{2+n-1} \\ (n-1)x^{2+n-1} \\ x^{2+n-2} \\ x^{2+n-2} \\ x^{2+n-2}\\ x^{2+n-2} \\ x^{2+n-2} \end{array}
& \begin{array}{c} x^{3+n-2} \\ x^{3+n-2} \\ x^{3+n-2} \\ 0\\ 0 \\ 0 \\ 0 \\ 0 \end{array}
& \begin{array}{c} x^{3+n-2} \\ x^{3+n-2} \\ x^{3+n-2} \\ 0\\ 0 \\ 0 \\ 0 \\ 0 \end{array}
& \begin{array}{c} (n-1)x^{3+n-1} \\ (n-1)x^{3+n-1} \\ (n-1)x^{3+n-1} \\ x^{3+n-2}\\ x^{3+n-2}\\ x^{3+n-2}\\ x^{3+n-2}\\ x^{3+n-2} \end{array}
& \begin{array}{c} x^{4+n-2} \\ x^{4+n-2} \\ x^{4+n-2} \\ 0 \\ 0 \\ 0 \\ 0 \\ 0 \end{array}
& \left).\begin{array}{c} \\ \\ \\ \\ \\ \\ \\  \end{array}\right.
\end{array},$$
and
$$T=\begin{array}{@{}r@{}c@{}c@{}c@{}c@{}c@{}c@{}c@{}c@{}l@{}}
& A_3 & B_3 & C_3 & D_3 & E_3 & F_3 & I_3 & J_3 \\
\left.\begin{array}{c} U_2 \\ V_2 \\ W_2 \\ X_2 \\ Y_2 \end{array}\right(
& \begin{array}{c} (n-1)x^{n+2} \\  (n-1)x^{n+2} \\ x^{n+1} \\ x^{n+1} \\ x^{n+1}   \end{array}
& \begin{array}{c} x^{n+1} \\  x^{n+1} \\ 0 \\ 0 \\ 0 \end{array}
& \begin{array}{c} x^{n+2} \\ x^{n+2} \\ 0 \\ 0 \\ 0  \end{array}
& \begin{array}{c} (n-1)x^{n+2} \\ (n-1)x^{n+2} \\ x^{n+1} \\ x^{n+1} \\ x^{n+1} \end{array}
& \begin{array}{c} x^{n+2} \\ x^{n+2} \\ 0 \\ 0 \\ 0 \end{array}
& \begin{array}{c} x^{n+2} \\ x^{n+2} \\ 0 \\ 0 \\ 0 \end{array}
& \begin{array}{c} (n-1)x^{n+3} \\ (n-1)x^{n+3} \\ x^{n+2} \\ x^{n+2} \\ x^{n+2} \end{array}
& \begin{array}{c} x^{n+3} \\ x^{n+3} \\ 0 \\ 0 \\ 0 \end{array}
& \left).\begin{array}{c} \\ \\ \\ \\ \\  \end{array}\right.
\end{array}$$
According to Corollary \ref{Close-walk}, it is clear that we only need to get the sum of the main diagonal elements of matrix $(R\cdot T)^2$.

Through tedious matrix multiplication, we obtain 
\begin{align*}
\sum_{i=1}^8((R\cdot T)^2)_{ii}&=\mathrm{tr}((R\cdot T)^2)
\\&=2x^{4n+3}+(2n+3)^2x^{4n+4}+2(n+1)^2(2n+3)x^{4n+5}+(n+1)^4x^{4n+6}.
\end{align*}
The proof is completed.
\end{proof}

\begin{cor}
Let $n\in \mathbb{N}$. Let $D(n)$ be the number of coronas of a diamond $D$ with side length $n$. Then we obtain the generating function of $D(n)$ as follows:
$$\sum_{n\geq 0}D(n)x^n=\frac{3x^4-13x^3+23x^2-7x+18}{(1-x)^5}.$$
\end{cor}

\section{An Extension for the Corona}

Now we consider an extension for the corona. In fact, we change the side lengths of regular hexagon $H$ or diamond $D$, but ensure that the interior angles remain unchanged. More precisely, the \emph{hexagon} $\overline{H}$ is defined by the left of Figure \ref{Tu13}, i.e., its three pairs of opposite sides have lengths $n_1,n_2,n_3\in \mathbb{N}$, respectively, and all interior angles are $120^\circ$. The \emph{diamond} $\overline{D}$ is defined by the right of Figure \ref{Tu13}, i.e., its two pairs of opposite sides have lengths $n_1,n_2\in \mathbb{N}$, respectively, and the interior angles are $60^\circ$, $120^\circ$, $60^\circ$, and $120^\circ$.

\begin{figure}[htp]
\centering
\includegraphics[width=11cm,height=6cm]{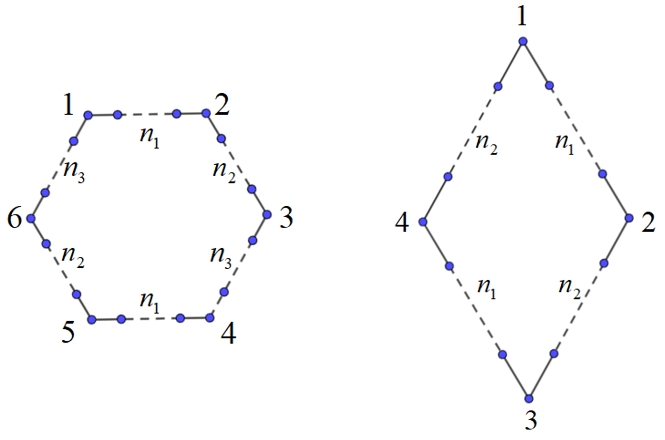}
\caption{A hexagon $\overline{H}$ and a diamond $\overline{D}$.}
\label{Tu13}
\end{figure}

Given a hexagon $\overline{H}$, a \emph{corona of a hexagon $\overline{H}$} is a lozenge tiling along the edges of $\overline{H}$ such that no additional lozenges are utilized. 
The number of coronas of a hexagon $\overline{H}$ is denoted by $\overline{H}(n_1,n_2,n_3)$.

\begin{thm}\label{ExtensionOne}
Let $n_1,n_2,n_3\in \mathbb{N}$. Let $\overline{H}(n_1,n_2,n_3)$ be the number of coronas of a hexagon $\overline{H}$.
Then there are only four cases in which the number of lozenges is used in a corona of a hexagon $\overline{H}$, namely $2(n_1+n_2+n_3)+3$, $2(n_1+n_2+n_3)+4$, $2(n_1+n_2+n_3)+5$, and $2(n_1+n_2+n_3)+6$.
Let $\overline{h}_i(n)$ be the number of corona tilings for $2(n_1+n_2+n_3)+2+i$ with $1\leq i\leq 4$. Then we have
\begin{align*}
&\overline{h}_1(n)=2,\ \ &&\overline{h}_3(n)=2(n_1+1)(n_2+1)(n_3+1)(n_1+n_2+n_3+3),
\\&\overline{h}_2(n)=(n_1+n_2+n_3+3)^2,\ \ &&\overline{h}_4(n)=(n_1+1)^2(n_2+1)^2(n_3+1)^2.
\end{align*}
Furthermore, we obtain
\begin{align*}
\overline{H}(n_1,n_2,n_3)=\sum_{i=1}^4\overline{h}_i(n)=((n_1+1)(n_2+1)(n_3+1)+(n_1+n_2+n_3+3))^2+2.
\end{align*}
\end{thm}
\begin{proof}
Similar to the proof of Theorem \ref{Them-A380346}, we define the following weighted adjacency matrix $M(n_i)$:
\begin{align*}
M(n_i)
=\left ( \begin{matrix}
0 & x^{3+n_i-2} & 0 & x^{2+n_i-2} & 0\\
0 & x^{3+n_i-2} & 0 & x^{2+n_i-2} & 0\\
0 & x^{3+n_i-2} & 0 & x^{2+n_i-2} & 0\\ 
x^{3+n_i-2} & (n_i-1)x^{3+n_i-1} & x^{4+n_i-2} & (n_i-1)x^{2+n_i-1} & x^{3+n_i-2} \\ 
x^{3+n_i-2}& (n_i-1)x^{3+n_i-1} & x^{4+n_i-2} & (n_i-1)x^{2+n_i-1} & x^{3+n_i-2}\\
\end{matrix} \right ).
\end{align*}
According to Corollary \ref{Close-walk}, it is clear that we only need to get the sum of the main diagonal elements of matrix $(M(n_1)\cdot M(n_2)\cdot M(n_3))^2$. Through matrix multiplication, it follows that 
\begin{align*}
\mathrm{tr}((M(n_1)\cdot M(n_2)\cdot M(n_3))^2)&=2x^{2(n_1+n_2+n_3)+3}+(n_1+n_2+n_3+3)^2x^{2(n_1+n_2+n_3)+4}
\\&\ +2(n_1+1)(n_2+1)(n_3+1)(n_1+n_2+n_3+3)x^{2(n_1+n_2+n_3)+5}
\\&\ +(n_1+1)^2(n_2+1)^2(n_3+1)^2x^{2(n_1+n_2+n_3)+6}.
\end{align*}
This completes the proof.
\end{proof}

Given a diamond $\overline{D}$, a \emph{corona of a diamond $\overline{D}$} is a lozenge tiling along the edges of $\overline{D}$ such that no additional lozenges are utilized. The number of coronas of a diamond $\overline{D}$ is denoted by $\overline{D}(n_1,n_2)$.

\begin{thm}\label{ExtensionTwo}
Let $n_1,n_2\in \mathbb{N}$. Let $\overline{D}(n_1,n_2)$ be the number of coronas of a diamond $\overline{D}$.
Then there are only four cases in which the number of lozenges is used in a corona of a diamond $\overline{D}$, namely $2(n_1+n_2)+3$, $2(n_1+n_2)+4$, $2(n_1+n_2)+5$, and $2(n_1+n_2)+6$.
Let $\overline{d}_i(n)$ be the number of corona tilings for $2(n_1+n_2)+2+i$ with $1\leq i\leq 4$. Then we have
\begin{align*}
&\overline{d}_1(n)=2,\ \ &&\overline{d}_3(n)=2(n_1+1)(n_2+1)(n_1+n_2+3),
\\&\overline{d}_2(n)=(n_1+n_2+3)^2,\ \ &&\overline{d}_4(n)=(n_1+1)^2(n_2+1)^2.
\end{align*}
Furthermore, we obtain
\begin{align*}
\overline{D}(n_1,n_2)=\sum_{i=1}^4\overline{d}_i(n)=((n_1+1)(n_2+1)+(n_1+n_2+3))^2+2.
\end{align*}
\end{thm}
\begin{proof}
Similar to the proof of Theorem \ref{Them-A380416}, we define the following two weighted adjacency matrices:
\begin{align*}
\overline{R}
=\left ( \begin{matrix}
(n_1-1)x^{2+n_1-1} & x^{3+n_1-2} & x^{3+n_1-2} & (n_1-1)x^{3+n_1-1} & x^{4+n_1-2}\\
(n_1-1)x^{2+n_1-1} & x^{3+n_1-2} & x^{3+n_1-2} & (n_1-1)x^{3+n_1-1} & x^{4+n_1-2}\\
(n_1-1)x^{2+n_1-1} & x^{3+n_1-2} & x^{3+n_1-2} & (n_1-1)x^{3+n_1-1} & x^{4+n_1-2}\\ 
x^{2+n_1-2} & 0 & 0 & x^{3+n_1-2} & 0 \\ 
x^{2+n_1-2} & 0 & 0 & x^{3+n_1-2} & 0 \\
x^{2+n_1-2} & 0 & 0 & x^{3+n_1-2} & 0 \\
x^{2+n_1-2} & 0 & 0 & x^{3+n_1-2} & 0 \\
x^{2+n_1-2} & 0 & 0 & x^{3+n_1-2} & 0 \\
\end{matrix} \right ),
\end{align*}
and
\begin{align*}
\overline{R}
=\left ( \begin{matrix}
(n_2-1)x^{n_2+2} & x^{n_2+1} & x^{n_2+2} & (n_2-1)x^{n_2+2} & x^{n_2+2} & x^{n_2+2} & (n_2-1)x^{n_2+3} & x^{n_2+3} \\
(n_2-1)x^{n_2+2} & x^{n_2+1} & x^{n_2+2} & (n_2-1)x^{n_2+2} & x^{n_2+2} & x^{n_2+2} & (n_2-1)x^{n_2+3} &  x^{n_2+3} \\
x^{n_2+1} & 0 & 0 & x^{n_2+1} & 0 & 0 & x^{n_2+2} & 0 \\ 
x^{n_2+1} & 0 & 0 & x^{n_2+1} & 0 & 0 & x^{n_2+2} & 0 \\ 
x^{n_2+1} & 0 & 0 & x^{n_2+1} & 0 & 0 & x^{n_2+2} & 0 \\
\end{matrix} \right ).
\end{align*}
According to Corollary \ref{Close-walk}, it is clear that we only need to get the sum of the main diagonal elements of matrix $(\overline{R}\cdot \overline{T})^2$.
Through matrix multiplication, we obtain 
\begin{align*}
\mathrm{tr}((\overline{R}\cdot \overline{T})^2)&=2x^{2n_1+2n_2+3}+(n_1+n_2+3)^2x^{2n_1+2n_2+4}
\\&\ +2(n_1+1)(n_2+1)(n_1+n_2+3)x^{2n_1+2n_2+5}+(n_1+1)^2(n_2+1)^2x^{2n_1+2n_2+6}.
\end{align*}
This completes the proof.
\end{proof}






\noindent
{\small \textbf{Acknowledgements:} 
The authors would like to thank the anonymous referee for valuable suggestions for improving the presentation. 
In June 2025, the second author's senior colleagues, Xinyu Xu, Chen Zhang, and Zihao Zhang, will complete their doctoral studies. Over the past few years, we have engaged in academic discussions and explored beautiful landscapes together. The author congratulates them on their successful graduation and dedicates this to the memory of our youthful years.

\end{document}